\documentclass[12pt]{amsart}
\usepackage{amsmath}
\usepackage{amscd}
\usepackage{amssymb}
\usepackage{amsfonts}
\newtheorem{theorem}{Theorem}[section]

\newtheorem{corollary}[theorem]{Corollary}

\newtheorem{remark}[theorem]{Remark}

\newcommand{\sn}{\mathrm{sn}}
\begin{document}
\title[Comparison geometry for integral radial Bakry-\'Emery Ricci tensor bounds]
{Comparison geometry for integral radial Bakry-\'Emery Ricci tensor bounds}

\author{Jia-Yong Wu}
\address{Department of Mathematics, Shanghai University, Shanghai 200444, China}
\email{wujiayong@shu.edu.cn}

\date{\today}

\begin{abstract}
In this paper we prove mean curvature comparisons and volume comparisons on a smooth metric
measure space when the integral radial Bakry-\'Emery Ricci tensor and the potential function
or its gradient are bounded. As applications, we prove diameter estimates and eigenvalue
estimates on smooth metric measure spaces. These results not only give a supplement of
the author's previous results under integral Bakry-\'Emery Ricci tensor bounds, but also
are generalizations of the Wei-Wylie's pointwise results.
\end{abstract}
\maketitle

\section{Introduction and main results}\label{Int1}
Classical comparison properties of the pointwise Ricci curvature condition, such as the mean
curvature comparison and the volume comparison, are basic theories for Riemannian manifolds.
See \cite{[Zhu]} for a survey and references therein. These comparison results were later
generalized to the integral Ricci tensor condition, which are briefly
described as follows. Given an $n$-dimensional complete Riemannian manifold $(M,g)$,
for each point $x\in M$, let $\lambda(x)$ be the smallest eigenvalue for the Ricci
curvature $\mathrm{Ric}:T_xM\to T_xM$, and let
\[
{\rm Ric}^H_-(x):=\left[(n-1)H-\lambda(x)\right]_+=\max\left\{0,(n-1)H-\lambda(x)\right\},
\]
the amount of the Ricci tensor below $(n-1)H$, where $H\in \mathbb{R}$. For any
real number $p>0$ and $R>0$, we consider
\[
\|{\rm Ric}^H_-\|_p(R):=\sup_{x\in M}\left(\int_{B(x,R)}(\mathrm{Ric}^H_-)^p\,
dv\right)^{\frac 1p},
\]
which measures the amount of Ricci tensor lying below $(n-1)H$ in the $L^p$ sense,
where $B(x,R)$ is the geodesic ball with radius $R$ and center $x$. It is easy to
see that $\|\mathrm{Ric}^H_-\|_p(R)\equiv0$ if and only if ${\mathrm{Ric}}\ge(n-1)H$.
Under certain assumption of $\|\mathrm{Ric}^H_-\|_p(R)$, Petersen and Wei
\cite{[PeWe],PeWe00} generalized classical comparison theorems to the integral
case. For more related results, we refer the reader to \cite {[Au], [Au2], [DPW], 
[DW], [DWZ], [Gal], [OSZW], [PeSp], SW, [ZZ]} and references therein.

\vspace{.1in}

In another direction, Wei and Wylie \cite{[WW]} extended comparison results of
Riemannian manifolds to smooth metric measure spaces. Recall that a complete smooth
metric measure space (SMMS for short) is a triple $(M,g,e^{-f}dv)$, where $(M,g)$ is an
$n$-dimensional Riemannian manifold, $dv$ is the volume element of the metric $g$, $f$ is a
smooth function on $M$ and $e^{-f}dv$ is the weighted volume element. The Bakry-\'Emery
Ricci tensor \cite{[BE]} and the $f$-Laplacian associated to $(M,g,e^{-f}dv)$ are
respectively defined by
\[
\mathrm{Ric}_f:=\mathrm{Ric}+\mathrm{Hess}\,f\quad\mathrm{and}\quad\Delta_f:=\Delta-\nabla f\cdot\nabla,
\]
where $\mathrm{Hess}$ and $\Delta$ are the Hessian and Laplacian with respect to
the metric $g$, respectively. The Bakry-\'Emery Ricci tensor and the $f$-Laplacian
are related by the generalized Bochner formula
\[
\Delta_f|\nabla u|^2=2|\mathrm{Hess}\,u|^2+2g(\nabla u,\nabla\Delta_f u)
+2\mathrm{Ric}_f(\nabla u, \nabla u)
\]
for $u\in C^\infty(M)$. The Bakry-\'Emery Ricci tensor is also related to the
gradient Ricci soliton defined by
\[
\mathrm{Ric}_f=\lambda g
\]
for some $\lambda\in\mathbb{R}$, which plays an important role in the singularities
of the Ricci flow \cite{[Ham]}. When $\mathrm{Ric}_f$ is bounded below and $f$ or
$|\nabla f|$ is bounded, Wei and Wylie \cite{[WW]} applied the generalized Bochner
formula to prove various weighted comparisons and topological results on $(M,g,e^{-f}dv)$.
Meanwhile, they expect that weighted comparisons can be extended to the case that
$\mathrm{Ric}_f$ is bounded below in the integral sense.

\vspace{.1in}

Inspired by the above work, the author \cite{[Wu]} generalized pointwise
weighted comparison theorems \cite{[WW]} to the integral Bakry-\'Emery Ricci
tensor setting. To be more precise, for each point $x\in(M,g,e^{-f}dv)$, we
consider a weighted geometric quantity
\[
{\mathrm{Ric}^H_f}_-(x):=\left[(n-1)H-\lambda(x)\right]_+=\max\{0,(n-1)H-\lambda(x)\},
\]
where $H\in\mathbb{R}$ and $\lambda(x)$ is the smallest eigenvalue of
$\mathrm{Ric}_f:T_xM\to T_xM$. When $\partial_rf\ge-a$ ($\partial_r:=\nabla r$)
for some constant $a\ge 0$, along a minimal geodesic segment $r$ from $x$,
we introduce a weighted $L^p$ norm of ${\mathrm{Ric}^H_f}_-$
\[
{\|{\mathrm{Ric}^H_f}_-\|_{p,a}}(R):=\sup_{x\in M}
\left(\int_{B(x,R)}|{\mathrm{Ric}^H_f}_-|^p\mathcal{A}_f e^{-at} dtd\theta_{n-1}\right)^{\frac 1p},
\]
where $\mathcal{A}_f(t,\theta)$ is the volume element of
$e^{-f}dv_g=\mathcal{A}_f(t,\theta)dt\wedge d\theta_{n-1}$ in polar coordinate,
and $d\theta_{n-1}$ is the volume element on unit sphere $S^{n-1}$. When
$\partial_rf\ge-a$, we easily see that ${\|{\mathrm{Ric}^H_f}_-\|_{p,a}}(R)\equiv0$
if and only if
${\mathrm{Ric}_f}\ge(n-1)H$. In \cite{[Wu]}, the author proved many weighted
comparison theorems on $(M,g,e^{-f}dv)$ when ${\|{\mathrm{Ric}^H_f}_-\|_{f,a}}(R)$ is
bounded and $\partial_rf\ge-a$. As applications, classical eigenvalue estimates,
Sobolev constant estimates and Myers' type theorems, etc were generalized to the
case of some assumptions of ${\|{\mathrm{Ric}^H_f}_-\|_{f,a}}(R)$ and $\partial_rf$;
see \cite{[Wu],[WaW],[LWZ]}. However, when $f$ is bounded, there seem to be
lack of effective comparison theorems under the integral Bakry-\'Emery Ricci tensor,
though some progress has been made in \cite{[Wu]}.

\vspace{.1in}

In this paper we will prove some comparison results on $(M,g,e^{-f}dv)$ when
the integral radial Bakry-\'Emery Ricci tensor is bounded and $f$ or
$\partial_r f$ is bounded. Our results are different from the case of \cite{[Wu]} and
seem to be new even in the manifold case. As applications, we prove some new
Myers' type theorems and eigenvalue estimates.

\vspace{.1in}

To state our results, we fix some notations. On SMMS $(M,g,e^{-f}dv)$, for any
point $x\in M$ and any $r(y):=d(y,x)$ a distance function from $x$ to $y$,
in geodesic polar coordinates $(r,\theta)$, we have another expression of
${\mathrm{Ric}^H_f}_-$:
\[
\rho(r,\theta):=\left[(n-1)H-\lambda(r,\theta)\right]_+,
\]
where $H\in\mathbb{R}$ and $\lambda(r,\theta)$ be the smallest eigenvalue for
$\mathrm{Ric}_f$ at the point $(r,\theta)$. Clearly,
\[
(n-1)H-\mathrm{Ric}_f(\partial_r,\partial_r)\le\rho(r,\theta)
\]
along that minimal geodesic segment from $x$; while $\rho(r,\theta)\equiv0$ along the
minimal geodesic segment $r$ if and only if ${\mathrm{Ric}}_f(\partial_r,\partial_r)\ge(n-1)H$.
Let $m^n_H$ denote the mean curvature of the geodesic sphere in the model space
$(M^n_H,g_H)$, the $n$-dimensional simply connected space with constant sectional
curvature $H$. For the weighted measure $e^{-f}dv$, we define the weighted mean
curvature
\[
m_f:=m-\partial_r f,
\]
which measures the relative rate of change of the weighted volume element of the
geodesic sphere, where $m$ is the mean curvature of the geodesic sphere in the
outer normal direction.

\vspace{.1in}

Let us first state weighted mean curvature comparisons on $(M,g,e^{-f}dv)$ along
the integral radial Bakry-\'Emery Ricci tensor.

\begin{theorem}[Mean Curvature Comparison]\label{Mainthm}
Let $(M,g,e^{-f}dv)$ be an $n$-dimensional smooth metric measure space with a base
point $x\in M$. Fix $H\in\mathbb{R}$.

(a) If $|f|\le k$ for some constant $k\geq 0$, along a minimal geodesic segment
$r$ from $x\in M$ (assume $r\le\frac{\pi}{4\sqrt{H}}$ when $H>0$), then
\[
m_f(r)\leq m^{n+4k}_H(r)+\int^r_0\rho(t,\theta)dt
\]
along that minimal geodesic segment from $x$. For the case $H>0$ and
$\frac{\pi}{4\sqrt{H}}\leq r\leq\frac{\pi}{2\sqrt{H}}$, then
\begin{equation}\label{pi/2}
m_f(r)\le\left(1+\frac{4k}{n-1}\cdot\frac{1}{\sin(2\sqrt{H}r)}\right)m^n_H(r)+\int^r_0\rho(t,\theta)dt
\end{equation}
along that minimal geodesic segment  from $x$.

(b) If $\partial_rf\ge-a$ for some constant $a\ge0$, along a minimal geodesic
segment $r$ from $x\in M$ (assume $r\le\frac{\pi}{2\sqrt{H}}$ when $H>0$), then
\[
m_f(r)\le m^n_H(r)+a+\int^r_0\rho(t,\theta)dt
\]
along that minimal geodesic segment from $x$. Equality holds if and only
if the radial sectional curvatures are equal to $H$ and $f(t)=f(x)-at$
for all $t<r$.
\end{theorem}

When $\rho=0$, we have $\mathrm{Ric}_f(\partial_r,\partial_r)\ge(n-1)H$ and
Theorem \ref{Mainthm} recovers Wei-Wylie's results \cite{[WW]}. When $k=0$,
Theorem \ref{Mainthm} reduces to the manifold cases, which seems to be
firstly appeared in the literature. The estimate \eqref{pi/2} will be used
in the Myers' type diameter estimate.

\vspace{.1in}

As in the classical case, weighted mean curvature comparisons have many applications.
First, we have weighted volume comparisons. On $(M,g,e^{-f}dv)$, the weighted volume
of the ball $B(x,r)$ is defined by
\[
V_f(x,r):=\int^r_0e^{-f}dv.
\]
Let $V^n_H(R)$ be the volume of
ball $B(O,R)$ in the model space $(M^n_H, g_H)$, where $O\in M^n_H$.
When $\partial_rf\ge-a$ for some constant $a\ge0$ along all minimal geodesic
segments from $x$, we introduce a new model volume according to constant $a$.
That is, let $V^a_H(R)$ be the $h$-volume of ball $B(O,R)$ in the pointed
smooth metric measure space
\[
M^n_{H,a}=(M^n_H, g_H,e^{-h}dv_{g_H},O),
\]
where $O\in M^n_H$ and $h(x)=-a\cdot d(O,x)$. Indeed we have
\[
V^a_H(R):=\int^R_0\int_{S^{n-1}}\mathcal{A}^a_H(t,\theta)\,d\theta_{n-1}dt=\int^R_0A^a_H(t)dt,
\]
where $\mathcal{A}^a_H(t,\theta)=e^{at}\mathcal{A}_H(t,\theta)$, $A^a_H(t)=e^{at}A_H(t)$,
$\mathcal{A}_H$ and $A_H$ are the volume element and the volume of the
geodesic sphere respectively in the model space $(M^n_H,g_H)$.

\begin{theorem}[Volume Comparison]\label{volcomp}
Let $(M,g,e^{-f}dv)$ be an $n$-dimensional complete smooth metric measure space
with a base point $x\in M$. Fix $H\in\mathbb{R}$. Assume that
\[
\int^{\infty}_0\rho(t,\theta)dt\le l
\]
along all minimal geodesic segments from $x\in M$, where $l\ge0$ is a constant.

(a) If $|f|\le k$ for some constant $k\ge 0$, then for $0<r\le R$
(assume $R\le\frac{\pi}{4\sqrt{H}}$ when $H>0$),
\[
\frac{V_f(x,R)}{V^{n+4k}_H(R)}\le\frac{V_f(x,r)}{V^{n+4k}_H(r)}
\exp\left\{\int^R_0\left(e^{c(n,k,H)lt}-1\right)
\frac{A^{n+4k}_H(t)}{V^{n+4k}_H(t)}dt\right\},
\]
where $c(n,k,H):=\frac{V(S^{n+4k-1})}{V(S^{n-1})}$ and $V(S^{n-1})$
is the area of the unit sphere $S^{n-1}\subset M^{n-1}_H$.

(b) If $\partial_rf\ge-a$ for some constant $a\ge 0$, along all minimal geodesic
segments from $x\in M$, then for $0<r\le R$ (assume $R\le\frac{\pi}{2\sqrt{H}}$
when $H>0$),
\[
\frac{V_f(x,R)}{V^a_H(R)}\le\frac{V_f(x,r)}{V^a_H(r)}
\exp\left\{\int^R_0\left(e^{lt}-1\right)\frac{A^a_H(t)}{V^a_H(t)}dt\right\}.
\]
Furthermore, when $r=0$, we have
\[
V_f(x,R)\le V^a_H(R)\exp\left\{-f(x)+\int^R_0\left(e^{lt}-1\right)\frac{A^a_H(t)}{V^a_H(t)}dt\right\}
\]
for $R\ge0$.
\end{theorem}

When $l=0$, Theorem \ref{volcomp} returns to Wei-Wylie's results \cite{[WW]}.
We remark that the term $\frac{V_f(x,r)}{V^{n+4k}_H(r)}$ $(k>0)$ in Theorem
\ref{volcomp} (a) blows up if $r\to 0$. If we let $r=1$, then
\begin{equation}\label{abvvg}
V_f(x,R)\le\frac{V_f(x,1)}{V^{n+4k}_H(1)} V^{n+4k}_H(R)\exp\left\{\int^R_0\left(e^{c(n,k,H)lt}-1\right)
\frac{A^{n+4k}_H(t)}{V^{n+4k}_H(t)}dt\right\}
\end{equation}
for $R\ge 1$. This estimate will be improved when $H<0$; see Theorem
\ref{abvc} in Section \ref{sec3}.

\vspace{.1in}

Next, we apply Theorem \ref{Mainthm} to give Myers' type diameter estimates,
which are regarded as generalizations of the Wei-Wylie's result \cite{[WW]}.

\begin{theorem}[Myers' Theorem]\label{Myerdiam}
Let $(M,g,e^{-f}dv)$ be an $n$-dimensional complete smooth metric measure space.
Fix $H\in\mathbb{R}^+$. Assume that
\[
\int^{\infty}_0\rho(t,\theta)dt\le l
\]
along all minimal geodesic segments from every point $p\in M$, where $l\ge0$ is a
constant.

(a) If $|f|\le k$ for some constant $k\ge 0$, then $M$ is
compact and
\[
\mathrm{diam}(M)\le\frac{\pi}{\sqrt{H}}+\frac{4k\sqrt{H}+2l}{(n-1)H}.
\]

(b) If $|\nabla f|\le a$ for some constant $a\ge 0$, then $M$ is compact and
\[
\mathrm{diam}(M)\le\frac{\pi}{\sqrt{H}}+\frac{2a+2l}{(n-1)H}.
\]
\end{theorem}

We point out that our integral assumption in Theorem \ref{Myerdiam} needs to hold
for \emph{every} point $p\in M$ and it seems to be a stronger condition.
In Section \ref{sec4}, we can apply the index form argument to get another
diameter estimate under a weaker assumption; see Theorem \ref{Myerdiam2}.

\vspace{.1in}

Finally, we apply volume comparisons to give a generalization of Cheng's eigenvalue
estimates \cite{[Cheng]}. On an $n$-dimensional SMMS $(M,g,e^{-f}dv)$, we assume
that $\partial_rf\ge-a$ for some constant $a\ge 0$, along all minimal geodesic
segments from a point $x_0\in M$. For any $H\in\mathbb{R}$ and $R>0$
($R\le\frac{\pi}{2\sqrt{H}}$ when $H>0$), we let $\lambda^D_1(B(x_0,R))$ be the
first eigenvalue of the $f$-Laplacian with the Dirichlet condition in
$B(x_0,R)\subseteq M$. We also let $\lambda^D_1(n,a,H,R)$ be the first
eigenvalue of the $h$-Laplacian $\Delta_h$, where $h(x):=-a\cdot d(\bar{x}_0,x)$,
with the Dirichlet condition in a metric ball $B(\bar{x}_0,R)\subseteq M^n_{H,a}$.
Then we have a weighted version of Petersen-Sprouse's result \cite{[PeSp]}.
\begin{theorem}[Cheng's Eigenvalue Estimate]\label{eigen}
Let $(M,g,e^{-f}dv)$ be an $n$-dimensional complete smooth metric measure space with
$\partial_rf\ge-a$ for some constant $a\ge 0$, along all minimal geodesic segments
from a point $x_0\in M$. Given $H\in\mathbb{R}$, $R>0$ (assume $R\le \frac{\pi}{2\sqrt{H}}$
when $H>0$), for every $\delta>0$, there exists an $\epsilon=\epsilon(n,a,H,R)$ such that if
\[
\int^{\infty}_0\rho(t,\theta)dt\le \epsilon
\]
along all minimal geodesic segments from the point $x_0\in M$, then
\[
\lambda^D_1(B(x_0,R))\le\left(1+\delta\right)\,\lambda^D_1(n,a,H,R).
\]
\end{theorem}
When $\rho\equiv0$ and $f$ is constant, Theorem \ref{eigen} returns to Cheng's
result \cite{[Cheng]}. In \cite{[Wu]}, the author proved another generalization
of Cheng's eigenvalue estimates, but this result is different from that case.
For the case $|f|\le k$, there seem to be essential obstacles to deriving Cheng's
eigenvalue estimates because volume comparisons in this case depend on
the volumes of \emph{higher} dimensional geodesic balls.

\vspace{.1in}

The rest of this paper is organized as follows. In Section \ref{sec2}, we study
mean curvature comparisons along the integral radial Bakry-\'{E}mery Ricci
tensor. In particular we prove Theorem \ref{Mainthm}. In Section \ref{sec3}, we
prove various volume comparisons, including Theorem \ref{volcomp} and the volume
doubling. In Section \ref{sec4}, we apply Theorem \ref{Mainthm} to prove Myers'
diameter estimates (Theorem \ref{Myerdiam}). We also apply the index form to
give another diameter estimate (Theorem \ref{Myerdiam2}). In Section \ref{sec5},
we apply the volume doubling to prove eigenvalue estimates (Theorem \ref{eigen}).

\vspace{.1in}

\textbf{Acknowledgement}.
The author would like to thank Homare Tadano for providing the manuscript \cite{Tad2}
and pointing out a mini omission in the proof of Theorem \ref{Myerdiam2}. He also thanks
the referee for a very careful reading of the paper and helpful suggestions. This work
was partially supported by the Natural Science Foundation of Shanghai (17ZR1412800).


\section{Mean curvature comparison}\label{sec2}
In this section, we will discuss mean curvature comparisons on $(M,g,e^{-f}dv)$ when the
integral radial Bakry-\'Emery Ricci tensor and $f$ or $\partial_r f$ are bounded. We
shall prove Theorem \ref{Mainthm}. The proof mainly uses the arguments of Petersen
and Wei \cite{[PeWe]}, and Wei and Wylie \cite{[WW]}. First, we give a rough estimate
on $m_f$ which will be used in the proof of Myers' type diameter estimates.

\begin{theorem}\label{Mainthm0}
Let $(M,g,e^{-f}dv)$ be an $n$-dimensional smooth metric measure space with a base
point $x\in M$. Fix $H\in\mathbb{R}$. Then given any minimal geodesic segment from
$x$ and $r_0>0$,
\[
m_f(r)\leq m_f(r_0)-(n-1)H(r-r_0)+\int^r_{r_0}\rho(t,\theta)dt
\]
for $r\ge r_0$. Equality holds for some $r>r_0$ if and only if
all the radial sectional curvatures are zero, $\mathrm{Hess}\,r\equiv 0$,
and $\partial_r\partial_rf=(n-1)H-\rho(r,\theta)$ along the geodesic from $r_0$ to $r$.
\end{theorem}

\begin{proof}[Proof of Theorem \ref{Mainthm0}]
Let $u=r(y)$, where $r(y)=d(y,x)$ is the distance function. It is well-known that
distance function $r$ is almost smooth on $M$ and also $|\nabla r|=1$ holds where
$r$ is smooth. Applying $u$ to the Bochner formula
\[
\Delta|\nabla u|^2=2|\mathrm{Hess}\,u|^2+2g(\nabla u,\nabla\Delta u)
+2\mathrm{Ric}(\nabla u, \nabla u)
\]
and using the fact $|\nabla r|=1$, we get
\begin{equation}\label{Boneq}
0=|\mathrm{Hess}\,r|^2+\partial_r(\Delta r)+\mathrm{Ric}(\partial_r,\partial_r),
\end{equation}
where $\partial_r=\nabla r$. Note that $\mathrm{Hess}\,r$ is the second fundamental
from of the geodesic sphere and $\Delta r=m$, the mean curvature of the geodesic sphere.
By the Schwarz inequality,
\begin{equation}\label{Bochineq}
m'\le-\frac{m^2}{n-1}-\mathrm{Ric}(\partial_r, \partial_r).
\end{equation}
Since $m_f:=m-\partial_r f$, i.e. $m_f=\Delta_f\, r$, then
\[
m_f'=m'-\partial_r\partial_r f,
\]
and hence
\[
m_f'\leq-\frac{m^2}{n-1}-\mathrm{Ric}_f(\partial_r, \partial_r).
\]
By the definition of $\rho(r,\theta)$, we get
\begin{equation}
\begin{aligned}\label{meancu}
 m_f'&\leq-\frac{m^2}{n-1}-(n-1)H-\rho(r,\theta)\\
&\leq-(n-1)H+\rho(r,\theta).
\end{aligned}
\end{equation}
Integrating this inequality from $r_0$ to $r$ gives the result.

To see the equality statement, suppose that
\[
m_f'=-(n-1)H+\rho(r,\theta)
\]
on an interval $[r_0,r]$, then from \eqref{meancu} we get $m=0$ (i.e. $\Delta r=0$).
We also have
\[
(n-1)H-\mathrm{Ric}_f(\partial_r,\partial_r)=\rho(r,\theta).
\]
So,
\[
 m_f'=-\partial_r\partial_r f=-\mathrm{Ric}_f(\partial_r,\partial_r)=-(n-1)H+\rho(r,\theta).
\]
This implies $\mathrm{Ric}(\partial_r,\partial_r)=0$. Then from
\eqref{Boneq} we have $\mathrm{Hess}\,r=0$, which implies
the sectional curvatures must be zero.
\end{proof}

\vspace{.1in}

In the following we will prove Theorem \ref{Mainthm}.
\begin{proof}[Proof of Theorem \ref{Mainthm}]
We start to prove part (a) of Theorem \ref{Mainthm}.
From \eqref{Bochineq}, we see that this inequality becomes equality if and only if the
radial sectional curvatures are
constant. So the mean curvature $m_H(r)$ of the $n$-dimensional model space satisfies
\[
m'_H=-\frac{m_H^2}{n-1}-(n-1)H,
\]
where
\[
m_H(r):=(n-1)\frac{\sn_H'(r)}{\sn_H(r)},
\]
and $\sn_H(r)$ is the unique function satisfying
\[
\sn_H''(r)+H\sn_H(r)=0
\]
with $\sn_H(0)=0$ and $\sn_H'(0)=1$. So
\begin{equation*}
\begin{aligned}
(m-m_H)'&\le-\frac{m^2-m_H^2}{n-1}+(n-1)H-\mathrm{Ric}(\partial r, \partial r)\\
&\le-\frac{m^2-m_H^2}{n-1}+\partial_r\partial_rf+\rho(r,\theta),
\end{aligned}
\end{equation*}
where we used the definition of $\rho$ in the second inequality. Then we compute that
\begin{equation*}
\begin{aligned}
{\left[\sn_H^2(m-m_H)\right]'}&=\sn_H^2\frac{2m_H}{n-1}\left(m-m_H\right)+\sn_H^2\left(-\frac{m^2-m_H^2}{n-1}+\partial_r\partial_rf+\rho(r,\theta)\right)\\
&=-\sn_H^2(r)\frac{(m-m_H)^2}{n-1}+\sn_H^2(r)\partial_r\partial_rf+\sn_H^2(r)\rho(r,\theta)\\
&\le\sn_H^2(r)\partial_r\partial_rf+\sn_H^2(r)\rho(r,\theta).
\end{aligned}
\end{equation*}
Integrating the above inequality from $0$ to $r$ yields
\[
\sn_H^2(r)m(r)\le\sn_H^2(r) m_H(r)+\int^r_0\sn_H^2(t)\partial_t\partial_tf(t)dt+\int^r_0\sn_H^2(t)\rho(t,\theta)dt.
\]
Integrating by parts on the above third term,
\begin{equation}\label{biji1}
\sn_H^2(r)m_f(r)\le\sn_H^2(r) m_H(r)-\int^r_0\partial_tf(t)(\sn_H^2)'(t)dt
+\int^r_0\sn_H^2(t)\rho(t,\theta)dt,
\end{equation}
where $m_f:=m-\partial_r f$. Integrating by parts on the above third term again,
\begin{equation}\label{biji}
\sn_H^2(r)m_f(r)\le\sn_H^2(r) m_H(r)-f(r)(\sn_H^2(r))'+\int^r_0 f(t)(\sn_H^2)''(t)dt
+\int^r_0\sn_H^2(t)\rho(t,\theta)dt.
\end{equation}
We see that if $H\le 0$, then $(\sn_H^2)''(t)\ge0$;
if $H>0$ and $0<r\le\frac{\pi}{4\sqrt{H}}$, then $(\sn_H^2)''(t)\ge0$.
Hence when $|f|\le k$, in any case, we have
\[
\sn_H^2(r)m_f(r)\le\sn_H^2(r) m_H(r)+2k(\sn_H^2(r))'+\int^r_0\sn_H^2(t)\rho(t,\theta)dt.
\]
Noticing that
\[
(\sn_H^2(r))'=2\sn_H(r)(\sn_H(r))'=\frac{2m_H(r)}{n-1}\sn_H^2(r)
\]
and $\sn_H^2(t)$ is increasing, we finally get
\[
m_f(r)\leq m^{n+4k}_H(r)+\int^r_0\rho(t,\theta)dt
\]
along that minimal geodesic segment from $x$. This proves the first inequality of theorem.

\vspace{.1in}

Next we prove the case $H>0$ and $\frac{\pi}{4\sqrt{H}}\le r\le\frac{\pi}{2\sqrt{H}}$.
We start with \eqref{biji} and give a delicate estimate. Since
$m_H(r)\ge 0$ for $\frac{\pi}{4\sqrt{H}}\leq r\leq\frac{\pi}{2\sqrt{H}}$, we observe that
\begin{equation*}
\begin{aligned}
-f(r)(\sn_H^2(r))'&=-f(r)\frac{2m_H(r)}{n-1}\sn_H^2(r)\\
&\le\frac{2k}{n-1}m_H(r)\sn_H^2(r).
\end{aligned}
\end{equation*}
Also,
\begin{equation*}
\begin{aligned}
\int^r_0 f(t)\cdot(\sn_H^2)''(t)dt&\le k\left(\int^{\frac{\pi}{4\sqrt{H}}}_0(\sn_H^2)''(t)dt
-\int^r_{\frac{\pi}{4\sqrt{H}}}(\sn_H^2)''(t)dt\right)\\
&=k\left(\frac{2}{\sqrt{H}}-\sn_H(2r)\right).
\end{aligned}
\end{equation*}
Substituting the above two estimates into \eqref{biji}, we have
\begin{equation*}
\begin{aligned}
\sn_H^2(r)m_f(r)&\le\left(1+\frac{2k}{n-1}\right)m_H(r)\sn_H^2(r)+k\left(\frac{2}{\sqrt{H}}-\sn_H(2r)\right)\\
&\quad+\int^r_0\sn_H^2(t)\rho(t,\theta)dt\\
&=\left(1+\frac{4k}{n-1}\right)\sn_H^2(r)\frac{m_H(r)}{\sin(2\sqrt{H}r)}+\int^r_0\sn_H^2(t)\rho(t,\theta)dt\\
&\le\left(1+\frac{4k}{n-1}\right)\sn_H^2(r)\frac{m_H(r)}{\sin(2\sqrt{H}r)}+\sn_H^2(r)\int^r_0\rho(t,\theta)dt.
\end{aligned}
\end{equation*}
Hence,
\[
m_f(r)\le\left(1+\frac{4k}{n-1}\cdot\frac{1}{\sin(2\sqrt{H}r)}\right)m_H(r)+\int^r_0\rho(t,\theta)dt
\]
which completes the second inequality of theorem. Hence Theorem \ref{Mainthm} (a)
follows.

\vspace{.1in}

Under Theorem \ref{Mainthm} (b) assumptions, we see that
\[
(\sn_H^2)'(t)=2\sn_H(t)(\sn_H)'(t)\ge0.
\]
So if $\partial_rf\ge-a$, from \eqref{biji1}, we have
\[
\sn_H^2(r)m_f(r)\le\sn_H^2(r) m_H(r)+a\int^r_0(\sn_H^2)'(t)dt
+\int^r_0\sn_H^2(t)\rho(t,\theta)dt
\]
and the third inequality of theorem follows.

To see the equality statement, assume that $\partial_rf\ge-a$ and
\[
m_f(r)=m^n_H(r)+a+\int^r_0\rho(t,\theta)dt
\]
for some $r$. Substituting them into \eqref{biji1},
\begin{equation*}
\begin{aligned}
a\,\sn_H^2(r)+\sn_H^2(r^2)\int^r_0\rho(t,\theta)dt
&\le-\int^r_0\partial_tf(t)(\sn_H^2)'(t)dt+\int^r_0\sn_H^2(t)\rho(t,\theta)dt\\
&\le a\int^r_0(\sn_H^2)'(t)dt+\int^r_0\sn_H^2(t)\rho(t,\theta)dt,
\end{aligned}
\end{equation*}
where we used $\partial_rf\ge-a$.
This implies $\rho(r,\theta)=0$ along that minimal geodesic segment $r$ from $x\in M$.
In other words, ${\mathrm{Ric}}_f(\partial_r,\partial_r)\ge(n-1)H$. Therefore the rigidity
follows from the rigidity for the Wei-Wylie's mean curvature comparison; see
Theorem 1.1 in \cite{[WW]}.
\end{proof}


\section{Volume comparison}\label{sec3}
In this section, we will apply mean curvature comparisons to prove volume comparisons
on $(M,g,e^{-f}dv)$ when the integral radial Bakry-\'Emery Ricci tensor is bounded
and $f$ or $\partial_rf$ is bounded.

\vspace{.1in}

On an $n$-dimensional SMMS $(M^n,g,e^{-f}dv_g)$, let
$\mathcal{A}_f(t,\theta)=e^{-f}\mathcal{A}(t,\theta)$ be the volume element of
the weighted volume form $e^{-f}dv_g=\mathcal{A}_f(t,\theta)dt\wedge d\theta_{n-1}$
in polar coordinate $(r,\theta)$, where $\mathcal{A}(t,\theta)$ is the standard volume
element of the metric $g$. Let
\[
A_f(x,r)=\int_{S^{n-1}}\mathcal{A}_f(r,\theta)d\theta_{n-1},
\]
be the weighted volume of the geodesic sphere $S(x,r)=\{y\in M|\,d(x,y)=r\}$, and let
$A_H(r)$ be the volume of the geodesic sphere $S(x,r)$ in the model space $(M^n_H,g_H)$,
the $n$-dimensional simply connected space with constant sectional curvature $H$.
Moreover, the weighted volume of the ball $B(x,r)=\{y\in M| d(x,y)\le r\}$ is defined by
\[
V_f(x,r)=\int^r_0A_f(x,t)dt.
\]

When $\partial_rf\ge-a$ for some constant $a\ge0$, along all minimal geodesic
segments from $x\in M$, we modify the usual model space $(M^n_H,g_H)$ to the weighted
model space $M^n_{H,a}=(M^n_H, g_H,e^{-h}dv_{g_H}, O)$, where $O\in M^n_H$, and
$h(x)=-a\cdot d(x,O)$. Let $\mathcal{A}^a_H$ be the
$h$-volume element in $M^n_{H,a}$. That is,
\[
\mathcal{A}^a_H(r)=e^{ar}\mathcal{A}_H(r),
\]
where $\mathcal{A}_H$ is the Riemannian volume element in $(M^n_H,g_H)$. The
corresponding $h$-volume of the geodesic sphere in the weighted
model space $M^n_{H,a}$ is defined by
\[
A^a_H(r)=\int_{S^{n-1}}\mathcal{A}^a_H(r,\theta)d\theta_{n-1}.
\]
The $h$-volume of the ball $B(O,r)\subset M^n_H$ is defined by
\[
V^a_H(r)=\int^r_0A^a_H(t)dt.
\]

\vspace{.1in}

In order to prove Theorem \ref{volcomp}, we first apply Theorem \ref{Mainthm} to
prove area comparisons of the geodesic spheres.
\begin{theorem}\label{areacomp}
Let $(M,g,e^{-f}dv)$ be an $n$-dimensional smooth metric measure space with base
point $x\in M$.  Fix $H\in\mathbb{R}$. Assume that
\[
\int^{\infty}_0\rho(t,\theta)dt\le l
\]
along all minimal geodesic segments from $x\in M$, where $l\ge0$ is a constant.

(a) If $|f|\le k$ for some constant $k\geq 0$, then for $0<r\le R$ (assume
$R\le\frac{\pi}{4\sqrt{H}}$ when $H>0$),
\begin{equation}\label{areacompineq}
\frac{A_f(x,R)}{A^{n+4k}_H(R)}\le e^{c(n,k,H)Rl}\frac{A_f(x,r)}{A^{n+4k}_H(r)}
\end{equation}
where $c(n,k,h):=\frac{V(S^{n+4k-1})}{V(S^{n-1})}$ and $V(S^{n-1})$
is the area of the unit sphere $S^{n-1}\subset M^{n-1}_H$.

(b) If $\partial_rf\ge-a$ for some constant $a\ge 0$, along all minimal geodesic
segments from $x\in M$, then for $0<r\le R$ (assume $R\le\frac{\pi}{2\sqrt{H}}$
when $H>0$),
\begin{equation}\label{areestiw}
\frac{A_f(x,R)}{A^a_H(R)}\le e^{Rl}\frac{A_f(x,r)}{A^a_H(r)}.
\end{equation}
\end{theorem}

\begin{proof}[Proof of Theorem \ref{areacomp}]
Applying
\[
\mathcal{A}'_f=m_f \mathcal{A}_f
\quad \mathrm{and} \quad
({\mathcal{A}^{n+4k}_H})'=m^{n+4k}_H \mathcal{A}_H,
\]
we compute that
\[
\frac{d}{dt}\left(\frac{\mathcal{A}_f(t,\theta)}{\mathcal{A}^{n+4k}_H(t)}\right)
=(m_f-m^{n+4k}_H)\frac{\mathcal{A}_f(t,\theta)}{\mathcal{A}^{n+4k}_H(t)}.
\]
Then by Theorem \ref{Mainthm} (a), we have
\begin{equation*}
\begin{aligned}
\frac{d}{dt}\left(\frac{A_f(x,t)}{A^{n+4k}_H(t)}\right)
&=\frac{1}{V(S^{n-1})}\int_{S^{n-1}}\frac{d}{dt}\left(\frac{\mathcal{A}_f(t,\theta)}{\mathcal{A}^{n+4k}_H(t)}\right)d\theta_{n-1}\\
&=\frac{V(S^{n+4k-1})}{V(S^{n-1})}\frac{1}{A^{n+4k}_H(t)}\int_{S^{n-1}}
\left(\int^t_0\rho(\tau,\theta)d\tau\right)\mathcal{A}_f(t,\theta)d\theta_{n-1}\\
&\le\frac{c(n,k,H)l}{A^{n+4k}_H(t)}\int_{S^{n-1}}\mathcal{A}_f(t,\theta)d\theta_{n-1}\\
&=c(n,k,H)l\frac{A_f(x,t)}{A^{n+4k}_H(t)},
\end{aligned}
\end{equation*}
where $c(n,k,H):=\frac{V(S^{n+4k-1})}{V(S^{n-1})}$ and $V(S^{n-1})$ is
the area of the unit sphere $S^{n-1}\subset M^{n-1}_H$. Here we used the relation
\[
A^{n+4k}_H(t)=\int_{S^{n+4k-1}}\mathcal{A}^{n+4k}_H(t)d\theta=V(S^{n+4k-1})\mathcal{A}^{n+4k}_H(t)
\]
in the above second equality. Separating variables and integrating from $r$ to $R$,
we immediately get \eqref{areacompineq}.

\vspace{.1in}

Next we shall prove \eqref{areestiw}. We apply
\[
\mathcal{A}'_f=m_f \mathcal{A}_f
\quad \mathrm{and} \quad
{\mathcal{A}^a_H}'=(m_H+a)\mathcal{A}^a_H
\]
to compute that
\[
\frac{d}{dt}\left(\frac{\mathcal{A}_f(t,\theta)}{\mathcal{A}^a_H(t)}\right)
=(m_f-m_H-a)\frac{\mathcal{A}_f(t,\theta)}{\mathcal{A}^a_H(t)}.
\]
Using this, by Theorem \ref{Mainthm} (b) and our theorem assumption, we estimate that
\begin{equation*}
\begin{aligned}
\frac{d}{dt}\left(\frac{A_f(x,t)}{A^a_H(t)}\right)
&=\frac{1}{V(S^{n-1})}\int_{S^{n-1}}\frac{d}{dt}\left(\frac{\mathcal{A}_f(t,\theta)}{\mathcal{A}^a_H(t)}\right)d\theta_{n-1}\\
&=\frac{1}{V(S^{n-1})}\int_{S^{n-1}}(m_f-m_H-a)\frac{\mathcal{A}_f(t,\theta)}{\mathcal{A}^a_H(t)}d\theta_{n-1}\\
&\le\frac{1}{V(S^{n-1})}\int_{S^{n-1}}\left(\int^{\infty}_0\rho(\tau,\theta)d\tau\right)
\frac{\mathcal{A}_f(t,\theta)}{\mathcal{A}^a_H(t)}d\theta_{n-1}\\
&\le l\frac{A_f(x,t)}{A^a_H(t)}.
\end{aligned}
\end{equation*}
Separating variables and integrating from $r$ to $R$,
we get \eqref{areestiw}.
\end{proof}

\vspace{.1in}

Similar to the argument of Petersen and Wei \cite{[PeWe]}, we will apply Theorem
\ref{areacomp} to complete the proof of Theorem \ref{volcomp}.
\begin{proof}[Proof of Theorem \ref{volcomp}]
We first prove part (a). Recall that
\[
\frac{V_f(x,r)}{V^{n+4k}_H(r)}=\frac{\int_0^rA_f(x,t)dt}{\int_0^rA^{n+4k}_H(t)dt}.
\]
So we have
\begin{equation}\label{deri}
\frac{d}{dr}\left(\frac{V_f(x,r)}{V^{n+4k}_H(r)}\right)
=\frac{A_f(x,r)\int_0^rA^{n+4k}_H(t)dt-A^{n+4k}_H(r)\int_0^rA_f(x,t)dt}{(V^{n+4k}_H(r))^2}.
\end{equation}
Notice that, by Theorem \ref{areacomp} (a), for $t\le r$,
\[
A_f(x,r)A^{n+4k}_H(t)-A^{n+4k}_H(r)A_f(x,t)
\le (e^{c(n,k,H)lr}-1)A^{n+4k}_H(r)A_f(x,t).
\]
Substituting this into \eqref{deri} yields
\[
\frac{d}{dr}\left(\frac{V_f(x,r)}{V^{n+4k}_H(r)}\right)\leq \left(e^{c(n,k,H)lr}-1\right)
\frac{A^{n+4k}_H(r)}{V^{n+4k}_H(r)}\left(\frac{V_f(x,r)}{V^{n+4k}_H(r)}\right).
\]
Separating variables and integrating from $r$ to $R$ ($r\le R$), we get
\begin{equation*}
\begin{aligned}
\frac{V_f(x,R)}{V^{n+4k}_H(R)}&\le\frac{V_f(x,r)}{V^{n+4k}_H(r)}
\exp\left\{\int^R_r\left(e^{c(n,k,H)lt}-1\right)
\frac{A^{n+4k}_H(t)}{V^{n+4k}_H(t)}dt\right\}\\
&\le\frac{V_f(x,r)}{V^{n+4k}_H(r)}\exp\left\{\int^R_0\left(e^{c(n,k,H)lt}-1\right)
\frac{A^{n+4k}_H(t)}{V^{n+4k}_H(t)}dt\right\}
\end{aligned}
\end{equation*}
and Theorem \ref{volcomp} (a) follows.

\vspace{.1in}

Next we shall prove Theorem \ref{volcomp} (b). The proof is very similar
to the arguments of part (a) in Theorem \ref{volcomp}. For the completeness we provide a
detailed proof. It is known that
\[
\frac{V_f(x,r)}{V^a_H(r)}=\frac{\int_0^rA_f(x,t)dt}{\int_0^rA^a_H(t)dt}.
\]
So we compute
\begin{equation}\label{deri2}
\frac{d}{dr}\left(\frac{V_f(x,r)}{V^a_H(r)}\right)
=\frac{A_f(x,r)\int_0^rA^a_H(t)dt-A^a_H(r)\int_0^rA_f(x,t)dt}{(V^a_H(r))^2}.
\end{equation}
By \eqref{areestiw}, we see that
\[
A_f(x,r)A^a_H(t)-A^a_H(r)A_f(x,t)\le (e^{lr}-1)A^a_H(r)A_f(x,t)
\]
for $t\le r$. Substituting this into \eqref{deri2} yields
\[
\frac{d}{dr}\left(\frac{V_f(x,r)}{V^a_H(r)}\right)\le \left(e^{lr}-1\right)
\frac{A^a_H(r)}{V^a_H(r)}\left(\frac{V_f(x,r)}{V^a_H(r)}\right).
\]
Separating variables and integrating from $r$ to $R$ ($r\le R$),  we have
\begin{equation*}
\begin{aligned}
\frac{V_f(x,R)}{V^a_H(R)}&\le\frac{V_f(x,r)}{V^a_H(r)}
\exp\left\{\int^R_r\left(e^{lt}-1\right)\frac{A^a_H(t)}{V^a_H(t)}dt\right\}\\
&\le\frac{V_f(x,r)}{V^a_H(r)}\exp\left\{\int^R_0\left(e^{lt}-1\right)\frac{A^a_H(t)}{V^a_H(t)}dt\right\}
\end{aligned}
\end{equation*}
and Theorem \ref{volcomp} (b) follows.
\end{proof}

\vspace{.1in}

The weighted volume comparisons immediately yield volume doubling properties
of smooth metric measure spaces.
\begin{corollary}[Volume Doubling]\label{corvde2}
Let $(M,g,e^{-f}dv)$ be an $n$-dimensional complete smooth metric measure space.

(a) Assume that $|f|\le k$ for some constant $k\geq 0$.
Given $H\in\mathbb{R}$, $\alpha>1$ and $R>0$(assume $R\leq \frac{\pi}{4\sqrt{H}}$ when $H>0$),
there is an $\epsilon=\epsilon(n,k,H,R,\alpha)$ such that if
\[
\int^{\infty}_0\rho(t,\theta)dt\le \epsilon,
\]
along all minimal geodesic segments from $x\in M$,
then for all $0<r_1<r_2\leq R$,
\[
\frac{V_f(x,r_2)}{V_f(x,r_1)}\leq \alpha\frac{V^{n+4k}_H(r_2)}{V^{n+4k}_H(r_1)}.
\]

(b) Assume that $\partial_rf\geq-a$ for some constant $a\geq 0$, along
all minimal geodesic segments from $x\in M$. Given $H\in\mathbb{R}$,
$\alpha>1$ and $R>0$ (assume $R\leq \frac{\pi}{2\sqrt{H}}$ when $H>0$),
there is an $\epsilon=\epsilon(n,a,H,R,\alpha)$ such that if
\[
\int^{\infty}_0\rho(t,\theta)dt\le \epsilon,
\]
along all minimal geodesic segments from $x\in M$, then for all
$0<r_1<r_2\leq R$,
\[
\frac{V_f(x,r_2)}{V_f(x,r_1)}\leq \alpha\frac{V^a_H(r_2)}{V^a_H(r_1)}.
\]
\end{corollary}
\begin{proof}[Proof of Corollary \ref{corvde2}]
We only prove part (a); the proof of part (b) is similar. Assume that
$\int^{\infty}_0\rho(t,\theta)dt\le l$ along all minimal geodesic segments from
$x\in M$, where $l\ge0$ is a constant. Since $|f|\le k$, by Theorem
\ref{volcomp} (a), for $0<r_1<r_2\le R$,
\begin{equation}\label{volcdineq}
\begin{aligned}
\frac{V_f(x,r_2)}{V^{n+4k}_H(r_2)}&\le\frac{V_f(x,r_1)}{V^{n+4k}_H(r_1)}
\exp\left\{\int^{r_2}_{r_1}\left(e^{c(n,k,H)l t}-1\right)
\frac{A^{n+4k}_H(t)}{V^{n+4k}_H(t)}dt\right\}\\
&\le\frac{V_f(x,r_1)}{V^{n+4k}_H(r_1)}
\exp\left\{\int^R_0\left(e^{c(n,k,H)l t}-1\right)
\frac{A^{n+4k}_H(t)}{V^{n+4k}_H(t)}dt\right\},
\end{aligned}
\end{equation}
where $c(n,k,H):=\frac{V(S^{n+4k-1})}{V(S^{n-1})}$. Notice that the right hand side of
integral quantity is finite (depending on $R$) because that
\[
\lim_{t\to 0}(e^{c(n,k,H)l t}-1)\frac{A^{n+4k}_H(t)}{V^{n+4k}_H(t)}=0.
\]
Set
\[
F(\sigma):=\int^R_0\left(e^{c(n,k,H)\sigma t}-1\right)\frac{A^{n+4k}_H(t)}{V^{n+4k}_H(t)}dt.
\]
We see that $F(0)=0$ and $e^{F(0)}=1$. Moreover, the function $F(\sigma)$ is continuous
with respect to the parameter $\sigma$. Therefore, for any $\alpha>1$, there exists a
number $\epsilon=\epsilon(n,k,H,R,\alpha)$ (as long as $\epsilon$ is small enough)
such that if $\int^{\infty}_0\rho(t,\theta)dt\le \epsilon$, then
\[
e^{F(\epsilon)}\le \alpha.
\]
Hence the conclusion follows.
\end{proof}

\vspace{.1in}

In the end of this section, we will give an absolute volume comparison when $H<0$
by modifying the argument of Jaramillo \cite{[Jar]}, which is an
improvement of \eqref{abvvg}. When $\rho\equiv 0$, this result returns to Jaramillo's
result \cite{[Jar]}.
\begin{theorem}\label{abvc}
Let $(M,g,e^{-f}dv)$ be an $n$-dimensional complete smooth metric measure space
with a base point $x\in M$. Fix $H<0$. Assume that
\[
\int^{\infty}_0\rho(t,\theta)dt\le l
\]
along all minimal geodesic segments from $x\in M$, where $l\ge0$ is a constant.
If $|f|\le k$ for some constant $k\geq 0$, then
\[
V_f(x,R)\le e^{3k}\int^R_0\mathcal{A}_H(t)e^{\cosh(2\sqrt{-H}t)+lt}dt
\]
for all $R\ge 0$.
\end{theorem}

\begin{proof}[Proof of Theorem \ref{abvc}]
Recall that in the course of proving Theorem \ref{Mainthm} (a), by \eqref{biji} and the
increase of $\sn_H^2(r)$, we indeed prove that
\[
m_f(r)\le m_H(r)-f(r)\frac{(\sn_H^2(r))'}{\sn_H^2(r)}+\int^r_0 f(t)\frac{(\sn_H^2)''(t)}{\sn_H^2(r)}dt
+\int^r_0\rho(t,\theta)dt
\]
along any a minimal geodesic segment from $x$, where $\sn_H(r)=\frac{1}{\sqrt{-H}}\sinh\sqrt{H}r$, since $H<0$.
Integrating the above inequality from $r_1$ to $r_2$ ($r_2\ge r_1$) gives
\begin{equation*}
\begin{aligned}
\int^{r_2}_{r_1}m_f(r)dr&\le\int^{r_2}_{r_1}m_H(r)dr-\int^{r_2}_{r_1}f(r)\frac{(\sn_H^2(r))'}{\sn_H^2(r)}dr
+\int^{r_2}_{r_1}\frac{1}{\sn_H^2(r)}\left[\int^r_0 f(t)(\sn_H^2)''(t)dt\right]dr\\
&\quad+\int^{r_2}_{r_1}\left(\int^r_0\rho(t,\theta)dt\right)dr.
\end{aligned}
\end{equation*}
Notice that
\begin{equation*}
\begin{aligned}
-&\int_{r_1}^{r_2}f(r)\frac{(\sn_H^2(r))'}{\sn_H^2(r)}dr+\int_{r_1}^{r_2}\frac{1}{\sn_H^2(r)}
\left[\int_0^r f(t)(\sn_H^2)''(t)dt\right]dr\\
&=-2\sqrt{-H}\int_{r_1}^{r_2}f(r)\coth\sqrt{-H}rdr
-2H\int_{r_1}^{r_2}{\mathrm{csch}}^2\sqrt{-H}r\left[\int_0^rf(t)\cosh 2\sqrt{-H}tdt\right]dr\\
&=-2\sqrt{-H}\int_{r_1}^{r_2}f(r)\coth\sqrt{-H}rdr
-2H\left[-\frac{\coth\sqrt{-H}r}{\sqrt{-H}}\int_0^rf(t)\cosh2\sqrt{-H}t dt\right]_{r_1}^{r_2}\\
&\quad-4H\int_{r_1}^{r_2}\frac{\coth\sqrt{-H}r}{\sqrt{-H}}f(r)\sinh^2\sqrt{-H}rdr
-2H\int_{r_1}^{r_2}\frac{\coth\sqrt{-H}r}{\sqrt{-H}}f(r)dr.
\end{aligned}
\end{equation*}
Using the assumption $|f|\le k$, we further have
\begin{equation*}
\begin{aligned}
-\int_{r_1}^{r_2}&f(r)\frac{(\sn_H^2(r))'}{\sn_H^2(r)}dr+\int_{r_1}^{r_2}\frac{1}{\sn_H^2(r)}
\left[\int_0^r f(t)(\sn_H^2)''(t)dt\right]dr\\
&\le k\coth\sqrt{-H}r_2\sinh(2\sqrt{-H}r_2)+k\coth\sqrt{-H}r_1\sinh(2\sqrt{-H}r_1)\\
&\quad+2k\left(\sinh^2\sqrt{-H}r_2-\sinh^2\sqrt{-H}r_1\right)\\
&=2k\left[\cosh(2\sqrt{-H}r_2)+1\right].
\end{aligned}
\end{equation*}
Therefore, for $r_1\le r_2$, we have
\[
\int^{r_2}_{r_1}m_f(r)dr\le\int^{r_2}_{r_1}m_H(r)dr+2k\left[\cosh(2\sqrt{-H}r_2)+1\right]+l(r_2-r_1),
\]
where we used $|f|\le k$ and $\int^{\infty}_0\rho(t,\theta)dt\le l$. This implies
\[
\ln\left(\frac{\mathcal{A}_f(r_2,\theta)}{\mathcal{A}_f(r_1,\theta)}\right)
\le\ln\left(\frac{\mathcal{A}_H(r_2)}{\mathcal{A}_H(r_1)}\right)
+2k\left[\cosh(2\sqrt{-H}r_2)+1\right]+l(r_2-r_1)
\]
for $r_1\le r_2$, and hence
\[
\mathcal{A}_f(r_2,\theta)\mathcal{A}_H(r_1)
\le\mathcal{A}_f(r_1,\theta)\mathcal{A}_H(r_2)e^{2k\left[\cosh(2\sqrt{-H}r_2)+1\right]+lr_2}.
\]
for all $r_1\le r_2$. Integrating both sides of the inequality over $S^{n-1}$
with respect to $\theta$ gives
\[
\mathcal{A}_H(r_1)\int_{S^{n-1}}\mathcal{A}_f(r_2,\theta)d\theta
\le\mathcal{A}_H(r_2) e^{2k\left[\cosh(2\sqrt{-H}r_2)+1\right]+lr_2}
\int_{S^{n-1}}\mathcal{A}_f(r_1,\theta)d\theta
\]
for $r_1\le r_2$. Then integrating both sides of the inequality with respect to $r_1$
from $0$ to $R_1$,
\[
V_H(r_1)\int_{S^{n-1}}\mathcal{A}_f(r_2,\theta)d\theta
\le V_f(x,R_1)\mathcal{A}_H(r_2)e^{2k\left[\cosh(2\sqrt{-H}r_2)+1\right]+lr_2}
\]
for $R_1\le r_2$. Finally integrating both sides of the inequality with respect to $r_2$
from $0$ to $R_2$,
\[
V_H(R_1)V_f(x,R_2)
\le V_f(x,R_1)\int^{R_2}_0\mathcal{A}_H(r_2)e^{2k\left[\cosh(2\sqrt{-H}r_2)+1\right]+lr_2}dr_2
\]
for $R_1\le R_2$. Namely,
\[
\frac{V_H(R_1)}{V_f(x,R_1)}
\le\frac{\int^{R_2}_0\mathcal{A}_H(r_2)e^{2k\left[\cosh(2\sqrt{-H}r_2)+1\right]+lr_2}dr_2}{V_f(x,R_2)}
\]
for $R_1\le R_2$. Letting $R_1\to 0$, the left hand side tends to $e^{f(x)}$ and hence
\begin{equation*}
\begin{aligned}
V_f(x,R_2)&\le e^{f(x)}\int^{R_2}_0\mathcal{A}_H(r_2)e^{2k\left[\cosh(2\sqrt{-H}r_2)+1\right]+lr_2}dr_2\\
&\le e^{3k}\int^{R_2}_0\mathcal{A}_H(r_2)e^{\cosh(2\sqrt{-H}r_2)+lr_2}dr_2
\end{aligned}
\end{equation*}
for all $R_2\ge0$. This finishes the proof.
\end{proof}


\section{Myers' theorem}\label{sec4}

In this section, we will discuss some Myers' type diameter estimates on $(M,g,e^{-f}dv)$
when the integral radial Bakry-\'Emery Ricci tensor and $f$ or $|\nabla f|$ are bounded.
First, we will apply mean curvature comparisons of Section \ref{sec2} to prove Theorem
\ref{Myerdiam}. The proof uses the excess function which is similar to the Wei-Wylie's
argument \cite{[WW]}; see also \cite{[Wu18]}.
\begin{proof}[Proof of Theorem \ref{Myerdiam}]
We first prove part (a). Choose two any points $p_1$ and $p_2$ in $(M,g,f)$
such that $d(p_1, p_2)\ge\frac{\pi}{\sqrt{H}}$ and set
\[
B:=d(p_1,p_2)-\frac{\pi}{\sqrt{H}}.
\]
Let
\[
r_1(x)=d(p_1,x) \quad \mathrm{and} \quad r_2(x)=d(p_2,x),
\]
and let $e(x)$ be the excess function for the points $p_1$ and $p_2$, that is,
\[
e(x):=d(p_1,x)+d(p_2,x)-d(p_1,p_2).
\]
The excess function measures how much the triangle inequality fails to be an equality.
By the triangle inequality, we obviously have $e(x)\ge 0$ and $e(\gamma(t))=0$, where
$\gamma$ is a minimal geodesic from $p_1$ to $p_2$. Hence $\Delta_f(e(\gamma(t))) \ge0$
in the barrier sense.
Let
\[
y_1=\gamma\left(\frac{\pi}{2\sqrt{H}}\right)
\quad\mathrm{and}\quad
y_2=\gamma\left(\frac{\pi}{2\sqrt{H}}+B\right).
\]
Then we see that $r_i(y_i)= \frac{\pi}{2\sqrt{H}}$, $i=1,2$. Furthermore, by the estimate
\eqref{pi/2} of Theorem \ref{Mainthm} and our assumption, we have
\begin{equation}\label{meanest1}
\begin{aligned}
\Delta_f(r_i(y_i))&\le 2k\sqrt{H}+\int^{\infty}_0\rho(t,\theta)dt\\
&\le 2k\sqrt{H}+l.
\end{aligned}
\end{equation}
Noticing that $r_1(y_2)>\frac{\pi}{2\sqrt{H}}$, we can not give an upper estimate
for $\Delta_f(r_1(y_2))$ by directly using Theorem \ref{Mainthm}.
But we can apply Theorem \ref{Mainthm0} and \eqref{meanest1} to get that
\begin{equation}\label{meanest2}
\begin{aligned}
\Delta_f(r_1(y_2))&\le 2k\sqrt{H}-B(n-1)H+\int^{\infty}_0\rho(t,\theta)dt\\
&\le 2k\sqrt{H}-B(n-1)H+l.
\end{aligned}
\end{equation}
Combining \eqref{meanest1} and \eqref{meanest2}, we get that
\begin{equation*}
\begin{split}
0\leq\Delta_f(e(y_2))
&=\Delta_f(r_1(y_2))+\Delta_f(r_2(y_2))\\
&\le 4k\sqrt{H}-B(n-1)H+2l,
\end{split}
\end{equation*}
which implies
\[
B\le\frac{4k\sqrt{H}+2l}{(n-1)H}
\]
and hence
\[
d(p_1, p_2)\le\frac{\pi}{\sqrt{H}}+\frac{4k\sqrt{H}+2l}{(n-1)H}.
\]
Since $p_1$ and $p_2$ are arbitrary two points, this completes the proof of part (a).

\vspace{.1in}

The proof of part (b) is almost the same as the part (a) and the main difference is that
we apply Theorem \ref{Mainthm} (b) instead of the estimate \eqref{pi/2}.
So we omit it here.
\end{proof}

\vspace{.1in}

In the end of this section, we will apply the index form technique to get another Myers'
type diameter estimate. In this case, the integral assumption is weaker than that of
Theorem \ref{Myerdiam} (a). The proof is inspired by the argument of Limoncu
\cite{[Lim]}; see also \cite{Tad}.

\begin{theorem}\label{Myerdiam2}
Let $(M,g,e^{-f}dv)$ be an $n$-dimensional complete smooth metric measure space.
Fix a point $p\in M$ and $H\in\mathbb{R}^+$. Assume that
\[
\int^{\infty}_0\rho(t,\theta)dt\le l
\]
along all minimal geodesic segments from the point $p$, where $l\ge0$ is a
constant. If $|f|\le k$ for some constant $k\geq 0$, then $M$ is
compact and
\[
\mathrm{diam}(M)\le\frac{2\pi}{\sqrt{H}}\sqrt{1+\frac{8k}{(n-1)\pi}+\frac{l^2}{(n-1)^2H\pi^2}}+\frac{2l}{(n-1)H}.
\]
\end{theorem}

We would like to point out that Tadano \cite{Tad2} also proved a Myers' type diameter estimate
for the integral radial Bakry-\'Emery Ricci tensor. But his curvature condition is
different from our case.

Before proving the theorem, let us recall some notations. Let $X, Y, Z$ be three smooth vector fields
on Riemannian manifold $(M,g)$. For any smooth function $f \in C^{\infty}(M)$, the gradient vector field
and Hessian of $f$ are defined by
\[
g(\nabla f,X)=df(X)
\quad\mathrm{and}\quad
\operatorname{Hess}f (X, Y)=g(\nabla_X\nabla f,Y),
\]
respectively. The Riemannian curvature tensor and the Ricci
curvature are defined by
\[
\operatorname{Rm}(X,Y)Z=\nabla_X \nabla_Y Z - \nabla_Y \nabla_X Z-\nabla_{[X, Y]} Z
\quad \mbox{and}\quad\operatorname{Ric}_g(X, Y)=\sum_{i=1}^ng(\operatorname{Rm}(e_i, X)Y, e_i),
\]
respectively, where $\{e_i\}_{i=1}^n$ denotes an orthonormal frame of $(M, g)$.
\begin{proof}[Proof of Theorem \ref{Myerdiam2}]
On $(M,g,e^{-f}dv)$, for the fixed point $p\in M$, let any point $q\in M$ and let
$\sigma$ be a minimizing unit speed geodesic segment from $p$ to $q$ of length $L$.
Consider a parallel orthonormal frame $\{e_1=\dot{\sigma}, e_2,...,e_n\}$ along
$\sigma$ and a smooth function $\phi\in C^\infty([0,L])$ such that $\phi(0)=\phi(L)$=0,
and we have
\[
I(\phi e_i,\phi e_i)=\int^L_0\left[g(\dot{\phi}e_i,\dot{\phi}e_i)
-g(\operatorname{Rm}(\phi e_i,\dot{\sigma})\dot{\sigma},\phi e_i)\right]dt,
\]
where $I(\cdot, \cdot)$ deontes the index form of the geodesic segment $\sigma$.
Summing $i$ from $1$ to $n$ in the above equality and using $g(\operatorname{Rm}(\dot{\sigma},\dot{\sigma})\dot{\sigma},\dot{\sigma})=0$,
we get
\[
\sum_{i=2}^n I(\phi e_i,\phi e_i)
=\int_0^L\left[(n-1)\dot{\phi}^2-\phi^2\operatorname{Ric}_g(\dot{\sigma},\dot{\sigma})\right]dt.
\]
According to the definition of $\rho$, we have
\begin{equation}
\begin{aligned}\label{keyineq}
\sum_{i = 2}^n I(\phi e_i, \phi e_i) & \leqslant \int_0^L \left[(n - 1)(\dot{\phi}^2 - H \phi^2) + \phi^2 \operatorname{Hess} f(\dot{\sigma}, \dot{\sigma})\right]dt+\int^L_0\phi^2\rho(t,\theta)dt \\
&=\int_0^L\left[(n-1)(\dot{\phi}^2-H\phi^2)+\phi^2g(\nabla_{\dot{\sigma}}\nabla f,\dot{\sigma})\right] dt+\int^L_0\phi^2\rho(t,\theta)dt\\
&=\int_0^L\left[(n-1)(\dot{\phi}^2-H\phi^2)+\phi^2\dot{\sigma}(g(\nabla f,\dot{\sigma}))\right]dt
+\int^L_0\phi^2\rho(t,\theta)dt,
\end{aligned}
\end{equation}
where we used the parallelism of the Riemannian metric $g$ and $\nabla_{\dot{\sigma}}\dot{\sigma}=0$
in the last equality. Along the geodesic segment $\sigma(t)$, we get that
\begin{align*}
\phi^2\dot{\sigma}\left(g(\nabla f,\dot{\sigma})\right)
&=\phi^2\frac{d}{dt}\left(g(\nabla f,\dot{\sigma})\right)\\
&=\frac{d}{dt}\left(\phi^2g(\nabla f,\dot{\sigma})\right)-2\phi\dot{\phi} g(\nabla f,\dot{\sigma})\\
&=\frac{d}{dt}\left(\phi^2 g(\nabla f,\dot{\sigma})\right)+2f\frac{d}{dt}\left(\phi\dot{\phi}\right)
-2\frac{d}{dt}\left(f\phi\dot{\phi}\right),
\end{align*}
where we used $g(\nabla f,\dot{\sigma})=\frac{df}{dt}(\sigma(t))$ in the last equality.
Then integrating the both sides of the above equality, we get
\begin{align*}
\int_0^L\phi^2\dot{\sigma}(g(\nabla f,\dot{\sigma}))dt
&=\phi^2g(\nabla f,\dot{\sigma}){\Big|}_0^L
+\int_0^L2f\frac{d}{dt}\left(\phi\dot{\phi}\right)dt-2f\phi\dot{\phi}{\Big|}_0^L\\
&=2\int_0^Lf\frac{d}{dt}\left(\phi\dot{\phi}\right)dt,
\end{align*}
where we used $\phi(0)=\phi(L)=0$ in the last equality. Since $|f|\leqslant k$
by the theorem assumption, then
\[
\int_0^L\phi^2\dot{\sigma}(g(\nabla f,\dot{\sigma}))dt
\le 2k\int_0^L\left|\frac{d}{dt}(\phi\dot{\phi})\right|dt.
\]
Substituting this into \eqref{keyineq}, we get that
\[
\sum_{i=2}^nI(\phi e_i,\phi e_i)
\le(n-1)\int_0^L(\dot{\phi}^2-H\phi^2)dt+2k\int_0^L\left|\frac{d}{dt}(\phi\dot{\phi})\right|dt
+\int^L_0\phi^2\rho(t,\theta)dt.
\]
If we take $\phi(t)=\sin(\frac{\pi t}{L})$, then
\[
\dot{\phi}(t)=\frac{\pi}{L}\cos\left(\frac{\pi t}{L}\right)
\quad\mathrm{and}\quad
\phi \dot{\phi}=\frac{\pi}{2L}\sin\left(\frac{2\pi t}{L}\right).
\]
We also know
\[
\int^{\infty}_0\sin^2\left(\frac{\pi t}{L}\right)\rho(t,\theta)dt\le\int^{\infty}_0\rho(t,\theta)dt\le l
\]
from our assumption. We collect these results together and the above estimate becomes
\begin{equation*}
\begin{aligned}
\sum_{i=2}^nI(\phi e_i,\phi e_i)&\le(n-1)\int_0^L\left[\frac{\pi^2}{L^2}\cos^2\left(\frac{\pi t}{L}\right)
-H\sin^2\left(\frac{\pi t}{L}\right)\right]dt \\
&\quad+2k\left(\frac{\pi}{L}\right)^2\int_0^L\left|\cos\frac{2\pi t}{L}\right|dt+l.
\end{aligned}
\end{equation*}
We simplify it and have that
\[
\sum_{i=2}^n I(\phi e_i,\phi e_i)\le-\frac{1}{2L}\left[(n-1)H L^2-(n-1)\pi^2-8\pi k\right]+l.
\]
Since $\sigma$ is a minimizing geodesic, then
\[
\sum_{i=2}^n I(\phi e_i,\phi e_i)\ge 0
\]
and we must have
\[
-\frac{1}{2L}\big[(n-1)HL^2-(n-1)\pi^2-8\pi k\big]+l\ge 0.
\]
This gives
\[
L\le\frac{\pi}{\sqrt{H}}\sqrt{1+\frac{8k}{(n-1)\pi}+\frac{l^2}{(n-1)^2H\pi^2}}+\frac{l}{(n-1)H}.
\]
Therefore for any two points $q_1, q_2\in M$, we have
\[
d(q_1,q_2)\le d(p,q_1)+d(p,q_2)\le 2L
\]
and the result follows.
\end{proof}
\begin{remark}
The index form argument also gives a Myers' type diameter estimate when the
integral radial Bakry-\'Emery Ricci tensor bounds and $\partial_rf$ is bounded
below along geodesics. To save the length of the paper, we omit them here.
\end{remark}

\section{Eigenvalue estimate}\label{sec5}

In this section we will apply the volume doubling of Section \ref{sec3}
(Corollary \ref{corvde2} (b)) to prove Theorem \ref{eigen} by following
the argument of \cite{[PeSp]} and \cite{[Wu]}.
\begin{proof}[Proof of Theorem \ref{eigen}]
Recall that $B(\bar{x}_0,R)$, where $R\le \frac{\pi}{2\sqrt{H}}$ when $H>0$
is a metric ball in the weighted model space $M^n_{H,a}$. Let
$\lambda^D_1(n,a,H,R)$ be the first eigenvalue of the $h$-Laplacian
$\Delta_h$ with the Dirichlet condition in $M^n_{H,a}$, where $h(x)=-a\cdot d(\bar{x}_0,x)$.
Let $u(x)=\phi(r)$
be the corresponding eigenfunction of $\lambda^D_1(n,a,H,R)$ such that
\[
\phi''+(m_H+a)\phi'+\lambda^D_1(n,a,H,R)\phi=0
\]
with $\phi(0)=1$ and $\phi(R)=0$. Since $\phi'<0$ on $[0,R]$, we see that $0\le\phi\le 1$.
Now we consider the Rayleigh quotient of $u(x)=\phi(d(x_0, x))$. We compute that
\begin{equation*}
\begin{aligned}
\int_{B(x_0,R)}|\nabla u|^2 e^{-f}dv&=\int_{S^{n-1}}\int^R_0(\phi')^2\mathcal{A}_f(t,\theta)\,dtd\theta_{n-1}\\
&=\int_{S^{n-1}}\left(\phi\phi'\mathcal{A}_f\Big|^R_0-\int^R_0\phi(\phi'\mathcal{A}_f)'\,dt\right)d\theta_{n-1}\\
&=-\int_{S^{n-1}}\int^R_0\phi(\phi''+m_f\phi')\mathcal{A}_f\,dtd\theta_{n-1}\\
&=-\int_{S^{n-1}}\int^R_0\phi(\phi''+(m^n_H+a)\phi')\mathcal{A}_f\,dtd\theta_{n-1}\\
&\quad-\int_{S^{n-1}}\int^R_0(m_f-m^n_H-a)\phi\phi'\mathcal{A}_f\,dtd\theta_{n-1}.
\end{aligned}
\end{equation*}
Noticing that
\[
\phi''+(m_H+a)\phi'=-\lambda^D_1(n,a,H,R)\phi
\]
so
\begin{equation*}
\begin{aligned}
\int_{B(x_0,R)}|\nabla u|^2 e^{-f}dv
&\le\lambda^D_1(n,a,H,R)\int_{S^{n-1}}\int^R_0\phi^2\mathcal{A}_f\,dtd\theta_{n-1}\\
&\quad+\int_{S^{n-1}}\int^R_0(m_f-m_H-a)_+|\phi'|\mathcal{A}_f\,dtd\theta_{n-1}.
\end{aligned}
\end{equation*}
Hence the Rayleigh quotient satisfies
\begin{equation}\label{Raineq}
Q:=\frac{\int_{B(x_0,R)}|\nabla u|^2 e^{-f}dv}{\int_{B(x_0,R)}u^2 e^{-f}dv}
\le\lambda^D_1(n,a,H,R)+\frac{\int_{S^{n-1}}\int^R_0(m_f-m^n_H-a)_+\,|\phi'|\mathcal{A}_f\,dtd\theta_{n-1}}
{\int_{S^{n-1}}\int^R_0\phi^2\mathcal{A}_f\,dtd\theta_{n-1}}.
\end{equation}

Next we will estimate the last term of the above inequality by choosing a proper
function $\phi$. Now we choose the first value $r=r(n,a,H,R)$ such that $\phi(r)=1/2$.
Then the last error term can be estimated as follows:
\begin{equation*}
\begin{aligned}
&\frac{\int_{S^{n-1}}\int^R_0(m_f-m^n_H-a)_+\,|\phi'|\mathcal{A}_f}
{\int_{S^{n-1}}\int^R_0\phi^2\mathcal{A}_f}\\
&\quad\le\frac{\left(\int_{S^{n-1}}\int^R_0(m_f-m^n_H-a)^2_+\,\mathcal{A}_f\right)^{\frac 12}
\left(\int_{S^{n-1}}\int^R_0|\phi'|^2\mathcal{A}_f\right)^{\frac 12}}
{\frac 12 V^{\frac 12}_f(x_0,r)\left(\int_{S^{n-1}}\int^R_0\phi^2\mathcal{A}_f\right)^{\frac 12}}\\
&\quad=2\left(\frac{\int_{S^{n-1}}\int^R_0(m_f-m^n_H-a)^2_+\,\mathcal{A}_f}{V_f(x_0,r)}\right)^{\frac 12}
\sqrt{Q},
\end{aligned}
\end{equation*}
where we used the Cauchy-Schwarz inequality and
\[
\int_{S^{n-1}}\int^R_0\phi^2\mathcal{A}_f\ge\frac{1}{4}V_f(x_0,r)
\]
in the above second inequality. On the other hand, if
$\int^{\infty}_0\rho(t,\theta)dt\le\epsilon(n,a,H,R)$ is very small
along all minimal geodesic segments from $x_0\in M$, by Corollary \ref{corvde2} (b), we have the
volume doubling
\[
\frac{V_f(x_0,R)}{V_f(x_0,r)}\leq 4\frac{V^a_H(R)}{V^a_H(r)}.
\]
Substituting this into the above error estimate,
\[
\frac{\int_{S^{n-1}}\int^R_0(m_f-m^n_H-a)_+\,|\phi'|\mathcal{A}_f}
{\int_{S^{n-1}}\int^R_0\phi^2\mathcal{A}_f}
\leq4\left(\frac{V^a_H(R)}{V^a_H(r)}\right)^{\frac 12}
\left(\frac{\int_{S^{n-1}}\int^R_0(m_f-m^n_H-a)^2_+\,\mathcal{A}_f}{V_f(x_0,R)}\right)^{\frac 12}\sqrt{Q}.
\]
Since $\int^{\infty}_0\rho(t,\theta)dt\le\epsilon(n,H,a,R)$ by the assumption of theorem, we observe that
\[
\int_{S^{n-1}}\int^R_0(m_f-m^n_H-a)^2_+\,\mathcal{A}_f\le\epsilon^2V_f(x_0,R).
\]
Hence we finally get
\[
\frac{\int_{S^{n-1}}\int^R_0(m_f-m^{n+4k}_H)_+\,|\phi'|\mathcal{A}_f}
{\int_{S^{n-1}}\int^R_0\phi^2\mathcal{A}_f}\le C(n,a,H,R)\epsilon\sqrt{Q}
\]
for some constant $C(n,a,H,R)$ depending on $n$, $a$, $H$ and $R$.
Substituting this estimate into \eqref{Raineq}, we have
\[
Q\le \lambda^D_1(n,a,H,R)+C(n,a,H,R)\epsilon\sqrt{Q},
\]
which implies the desired result.
\end{proof}

\textbf{Data availability statement} Data sharing not applicable to this article as
no datasets were generated or analysed during the current study.


\end{document}